\documentclass[11pt]{amsart}

\usepackage{amsfonts, amstext, amsmath, amsthm, amscd, amssymb}
\usepackage{epsfig, graphics, psfrag}
\usepackage{color}
\usepackage{url}
\usepackage[margin=1.2in]{geometry}
\usepackage{adjustbox}

\let\oldmarginpar\marginpar
\renewcommand\marginpar[1]{\oldmarginpar[\raggedleft\footnotesize #1]%
{\raggedright\footnotesize #1}}

\newcommand{\Z}{{\mathbb{Z}}}

\newcommand{\C}{{\mathbb{C}}}
\newcommand{\QQ}{{\mathbb{Q}}}

\theoremstyle{plain}
\newtheorem{theorem}{Theorem}[section]
\newtheorem{corollary}[theorem]{Corollary}
\newtheorem{lemma}[theorem]{Lemma}
\newtheorem{proposition}[theorem]{Proposition}

\newtheorem{conjecture}[theorem]{Conjecture}

\newtheorem*{namedtheorem}{\theoremname}
\newcommand{\theoremname}{testing}

\theoremstyle{definition}
\newtheorem{definition}[theorem]{Definition}
\newtheorem{remark}[theorem]{Remark}

\title[A quantum obstruction to cosmetic surgeries]{A quantum obstruction to purely cosmetic surgeries}

\author{Renaud Detcherry}
\address{Institut de Math\'ematiques de Bourgogne, UMR 5584 CNRS, Universit\'e Bourgogne Franche-Comt\'e, F-2100 Dijon, France \newline
         {\tt \url{http://detcherry.perso.math.cnrs.fr}}
         }
\email{renaud.detcherry@u-bourgogne.fr}

\begin{document}

\date{\today}

\begin{abstract}We present new obstructions for a knot $K$ in $S^3$ to admit purely cosmetic surgeries, which arise from the study of Witten--Reshetikhin--Turaev invariants at fixed level, and can be framed in terms of the colored Jones polynomials of $K.$
\\ In particular, we show that if $K$ has purely cosmetic surgeries then the slopes of the surgery are of the form $\pm \frac{1}{5k},$ except if $J_K(e^{\frac{2i\pi}{5}})=1,$ where $J_K$ is the Jones polynomial of $K.$ For any odd prime $r\geq 5,$ we also give an obstruction for $K$ to have a $\pm \frac{1}{k}$ surgery slope with $r\nmid k$ that involves the values of the first $\frac{r-3}{2}$ colored Jones polynomials of $K$ at an $r$-th root of unity. We verify the purely cosmetic surgery conjecture for all knots with at most $17$ crossings.
\end{abstract}


\maketitle

\footnotetext{This research has been funded by the project “ITIQ-3D” of the R\'egion Bourgogne Franche–Comt\'e and was carried out at the Institut de Math\'ematiques de Bourgogne. The IMB receives support from the EIPHI Graduate School (contract ANR-17-EURE-0002)}
\section{Introduction}
Given an oriented $3$-manifold $M$ with $\partial M=\mathbb{T}^2,$ a pair of slopes $s \neq s'$ on the boundary is said to form a \textit{cosmetic surgery pair} if the Dehn fillings $M(s)$ and $M(s')$ are homeomorphic, and a \textit{purely cosmetic surgery pair} if $M(s)$ and $M(s')$ are homeomorphic as oriented $3$-manifolds. We will write $M(s) \simeq M(s')$ when the two manifolds $M(s)$ and $M(s')$ are homeomorphic as oriented $3$-manifolds. 
\\
\\ In the case where $M$ is the complement $E_K$ of a non-trivial knot $K$ in $S^3,$ the cosmetic surgery conjecture, first formulated by Gordon in \cite{Gor91}, asserts that:
\begin{conjecture}(Cosmetic Surgery Conjecture)\label{conj:cosm_surg} If $K$ is a non-trivial knot in $S^3$ and $s \neq s'$ are two slopes, then $E_K(s) \not\simeq E_K(s').$
\end{conjecture}
In other words, no non-trivial knot in $S^3$ has a purely cosmetic surgery pair. In the more general setting of a $3$-manifold with torus boundary $M,$ the cosmetic surgery conjecture says that $M(s)\simeq M(s')$ if and only if there is a positive self-homeomorphism of $M$ which sends the slope $s$ to $s'.$
\\
\\ Cosmetic surgery pairs which are not purely cosmetic are called \textit{chirally cosmetic}. Chirally cosmetic pairs do exist: indeed, for any amphichiral knot, any pair of opposite slopes will form a chirally cosmetic pair. Moreover, some chirally cosmetic pairs on the right-hand trefoil knot, and more generally, $(2,n)$-torus knots, were found by Mathieu \cite{Mat92}.
\\
\\ The first result on Conjecture \ref{conj:cosm_surg} was given by Boyer and Lines \cite{BL90}. Using surgery formulas for the  Casson--Walker and Casson--Gordon invariants, they proved the conjecture for any knot $K$ such that $\Delta_K''(1)\neq 0,$ where $\Delta_K$ is the Alexander polynomial of $K.$ The next finite type invariant of knots also gives an obstruction: Ichihara and Wu showed \cite{IW16} that if $K$ admits purely cosmetic surgeries then $J_K'''(1)=0$ where $J_K$ is the Jones polynomial.
\\ Later on, much of the progress on Conjecture \ref{conj:cosm_surg} has been brought by studying Heegaard Floer homology of knots. First, it was proved than genus one knots do not have purely cosmetic surgeries \cite{Wan06}. Some conditions on the set of possible slopes in a purely cosmetic pair were established: it was first shown that the slopes in a purely cosmetic pair for a knot $K\subset S^3$ have opposite signs \cite{OS11} \cite{Wu11}, then Ni and Wu showed that the slopes must actually be opposite and furthermore that the Ozv\'ath-Szab\'o-Rasmussen $\tau$ invariant of $K$ must vanish \cite{NW15}. Also, the Heegaard Floer homology $\widehat{HFK}(K)$ must satisfy some other constraints \cite{NW15}\cite{Gai17}. 
\\ Then, the work of Futer, Purcell and Schleimer \cite{FPS19} showed for any hyperbolic $3$-manifold with torus boundary, it can be algorithmically checked whether it admits a purely cosmetic pair. Their result came out of a different direction, using hyperbolic geometry to bound the length of slopes in a purely cosmetic pair. With their method they proved the conjecture for knots with less than $15$ crossings.
\\ Using a new method to express the Heegaard Floer homology of surgeries, Hanselman put further restrictions on the surgery slopes \cite{Han19}. Similarly to Futer, Purcell and Schleimer's result, those conditions restrict the set of possible slopes to a finite set. Hanselman's bounds seem to be more powerful in practice, with the downside that so far they apply only to the case of knots in $S^3.$ 
\\ Indeed, Hanselman was able to use his results to show that no knot with prime summands with less than $16$ crossings has a purely cosmetic surgery pair. 
\\ Finally, let us mention that, analyzing JSJ decompositions, Tao has proved that composite knots \cite{Tao19} and cable knots \cite{Tao18} do not have purely cosmetic surgeries. It reduces the conjecture to the case of prime knots.
\\ We will now describe some of Hanselman's main theorem from \cite{Han19} in more detail below, as our results will build upon his work.
\\
\\ For $K$ a knot in $S^3,$ let $g(K)$ be its Seifert genus. Let $\widehat{HFK}(K)$ be its Heegaard Floer knot homology, which is bigraded with Alexander grading $A$ and Maslov grading $\mu.$ We define the $\delta$-grading by $\delta=A-\mu,$ and let the Heegaard Floer thickness of $K$ be 
$$th(K)=\mathrm{max}\lbrace\delta(x) \ | x\neq 0 \in \widehat{HFK}(K)\rbrace -\mathrm{min}\lbrace\delta(x) \ | x\neq 0 \in \widehat{HFK}(K)\rbrace$$
Said differently, $th(K)+1$ is the number of diagonals on which the Heegaard Floer homology of $K$ is supported.
\begin{theorem}\label{thm:hanselman}\cite{Han19} Let $K$ be a non-trivial knot in $S^3$ and $s\neq s'$ be slopes such that $E_K(s)\simeq E_K(s')$ then:
\begin{itemize}
\item[-]The pairs of slopes $\lbrace s,s' \rbrace$ is either $\lbrace \pm 2 \rbrace$ or of the form $\lbrace \pm \frac{1}{k} \rbrace$ for some non-negative integer $k.$
\item[-]If $\lbrace s,s' \rbrace=\lbrace \pm 2 \rbrace$ then $g(K)=2.$
\item[-]If $\lbrace s,s' \rbrace=\lbrace \pm \frac{1}{k} \rbrace$ then
$$k \leqslant \frac{th(K)+2g(K)}{2g(K)(g(K)-1)}$$
\end{itemize}
\end{theorem}  
\begin{remark}Actually, Hanselman's proof also show that there is an integer $n_s$ which can computed from  $\hat{HFK}(K)$ such that if $E_K(\tfrac{1}{k})\simeq E_K(-\tfrac{1}{k})$ then $k=n_s.$
\end{remark}
Despite the success of Heegaard Floer homology, it is generally believed that other methods will be needed to study the cosmetic surgery conjecture. Indeed, Hanselman found $337$ knots $K \subset S^3$ for which Heegaard Floer homology does not distinguish between $E_K(s)$ and $E_K(-s)$ for $s=1$ or $2.$
\\
\\ Invariants that have nice surgery expressions are natural candidates for applications to the cosmetic surgery conjecture. Therefore, the author tried to investigate what information the Witten--Reshetikhin--Turaev invariants have to offer about cosmetic surgeries. As WRT invariants are part of TQFTs, they admit some natural surgery expressions. The case of the so-called $\mathrm{SO}_3$ WRT invariants at level $5$ seems to give the most straightforward condition. We find:
\begin{theorem}\label{thm:main_thm} Let $K\subset S^3$ be a knot, let $J_K$ be its Jones polynomial, and assume that $K$ has a purely cosmetic surgery pair $\lbrace \pm s \rbrace.$
\\ Then $s$ is of the form $\pm \frac{1}{5k}$ for some $k \in \Z,$ unless $J_K(e^{\frac{2i\pi}{5}})=1.$
\end{theorem}
At this point, the role of the $5$-th root of unity, and multiples of $5$ in the denominator of cosmetic surgery slopes  may seem mysterious. This is mostly a matter of exposition: we will also get similar conditions out of $\mathrm{SO}_3$ WRT TQFTs at different levels than $5$; with the drawback that they involve several colored Jones polynomials of $K.$ Here, let us write $J_{K,n}$ for the $n$-th normalized colored Jones polynomial of $K,$ with the convention that $J_{K,1}=1$ and $J_{K,2}$ is the Jones polynomial.
\begin{theorem}\label{thm:main_thm2}
Let $r \geqslant 5$ be an odd prime, $\zeta_r$ be a primitive $r$-th root of unity and $K$ be a non-trivial knot. Let also $[n]=\frac{\zeta_r^n-\zeta_r^{-n}}{\zeta_r-\zeta_r^{-1}}.$ There exists a finite set $F_r$ of nonzero vectors in $\C^{\frac{r-1}{2}},$ with cardinality $|F_r|\leqslant \frac{r+1}{2}$ such that if $K$ has a purely cosmetic surgery pair, then either the slopes are of the form $\lbrace \pm \frac{1}{rk}\rbrace$ with $k\in \Z,$ or the vector 
$$v_K=\begin{pmatrix}
1 \\ -[2]J_{K,1}(\zeta_r^2) \\ [3]J_{K,2}(\zeta_r^2) \\ \vdots \\ (-1)^{\frac{r-3}{2}}[\frac{r-1}{2}]J_{K,\frac{r-3}{2}}(\zeta_r^2)
\end{pmatrix}$$
is orthogonal to an element of $F_r.$
\end{theorem}
\begin{remark}The vectors in $F_r$ can actually be naturally be associated with pairs $\lbrace \pm 2 \rbrace$ or $\lbrace \pm \frac{1}{\underline{k}}\rbrace$ where $\underline{k}$ is a non-zero element of $\Z/r\Z.$ As the proof of Theorem \ref{thm:main_thm} will show, if $E_K(2)\simeq E_K(-2)$ then $v_K$ is orthogonal to $f_{\lbrace \pm 2 \rbrace}$ and if $E_K(\tfrac{1}{k})\simeq E_K(-\tfrac{1}{k})$ then $v_K$ is orthogonal to $f_{\lbrace \frac{1}{\underline{k}}\rbrace}$ where $\underline{k}$ is the class of $k$ mod $r.$

In particular, if one were able to compute the invariant $n_s$ in Hanselman's theorem for a knot $K,$ then it would suffice to test orthogonality of $v_K$ with $f_{\lbrace \pm 2 \rbrace}$ and with $f_{\lbrace \pm \frac{1}{\underline{n_s}}\rbrace}$ to rule out a purely cosmetic surgery pair for $K.$
\end{remark}
\begin{remark}\label{rk:level3} While the $\mathrm{SO}_3$ WRT invariant at level $r=3$ exists, it is entirely determined by homological data. It turns out that the level $r=5$ is the first one where the $\mathrm{SO}_3$ WRT invariants give non-trivial information about cosmetic surgeries.
\end{remark} 
Let us now restrict to the simpler condition we get from the case $r=5$ and focus on the corollaries of Theorem \ref{thm:main_thm}.

 Theorem \ref{thm:main_thm} combined with Hanselman's results implies that if a knot $K$ has a purely cosmetic surgery pair and $J_K(e^{\frac{2i\pi}{r}})\neq 1,$ then $K$ must have a rather large crossing number:
\begin{corollary} If $K$ is a non-trivial knot with at most $31$ crossings and such that $J_K(e^{\frac{2i\pi}{5}})\neq 1,$ then $K$ has no purely cosmetic surgery pair.
\end{corollary}
\begin{proof}
By Theorem \ref{thm:main_thm}, if $K$ has a purely cosmetic surgery pair then the slopes must be of the form $\pm \frac{1}{5k},$ so the denominator is at least $5.$
\\ By Hanselman's inequality and the fact that $g(K)\geqslant 2$ by \cite{Wan06}, this implies that $th(K)\geqslant 16.$
\\ However, Lowrance proved that $th(K)\leqslant g_T(K),$ where $g_T(K)$ is the Turaev genus of $K$ \cite{Low08}.
\\ But for any knot $K,$ the Turaev genus $g_T(K)$ is bounded above by $c(K)/2$ where $c(K)$ is the crossing number of $K.$ (see for example \cite[Proposition 2.4]{Abe09}).
\\ Thus a knot with $c(K)\leqslant 31$ crossings has thickness at most $15,$ and can not have a purely cosmetic surgery pair if $J_K(e^{\frac{2i\pi}{5}})\neq 1.$  
\end{proof}
Let us note that according to Hanselman's computations \cite{Han19} a prime knot with at most $16$ crossings has thickness at most $2,$ so the inequality $th(K) \geqslant \frac{c(K)}{2}$ we used seems to be far from optimal.

 The literature has been particularly interested in the special case of alternating knots. In that case, hypothetical alternating counterexamples of the conjecture have been shown to have a very special form. Indeed, Hanselman showed in \cite[Theorem 3]{Han19} that an alternating knot $K$ (or, more generally, a \textit{thin knot}, which has $th(K)=0$) with a purely cosmetic surgery must have signature $0$ and Alexander polynomial $\Delta_K(t)=nt^2-4nt+(6n+1)-4nt^{-1}+nt^{-2}$ for some $n\in \Z.$ As a corollary of our main result, we can put an extra condition on the Jones polynomial of $K$:
\begin{corollary}\label{cor:alternating} If $K$ is an alternating (or thin) knot with a purely cosmetic surgery, then $J_K(e^{\frac{2i\pi}{5}})=1.$
\end{corollary}
\begin{proof}
It is also part of \cite[Theorem 3]{Han19} that the only possible purely cosmetic surgery pairs of thin knots are the pairs $\lbrace \pm 1\rbrace$ or $\lbrace \pm 2\rbrace.$ By Theorem \ref{thm:main_thm}, $K$ must satisfy $J_K(e^{\frac{2i\pi}{5}})=1.$
\end{proof}
Finally, let us discuss some numerical estimates of the strength of this obstruction. Knots with $J_K(e^{\frac{2i\pi}{5}})=1$ seem to become increasingly rare when the number of crossings increases. Using a census of all $9,755,328$ prime knots with at most $17$ crossings and their Jones polynomials generated with the Regina software\cite{Reg}, we found that there are only $97$ that have $J_K(e^{\frac{2i\pi}{5}})=1.$ Following Regina's notation, those are the knots:
\begin{small}
\begin{eqnarray*}8nt_{1}, 9at_{1}, 11ah_{001}, 11at_{1}, 12ah_{0001}, 12ah_{0002}, 12nh_{003}, 
13ah_{0002}, 13ah_{0004}, 13ns_{2},
\\ 14ah_{00003}, 14ah_{00005}, 14ah_{00010}, 14ah_{00012},  14ah_{00013}, 14ah_{00017}, 14ah_{00025}, 14nh_{00007}, 14nh_{00042},
\\ 15ah_{00005}, 15ah_{00007}, 15ah_{00044}, 15ah_{00049}, 15ah_{00050}, 15ah_{00070}, 15ah_{00072}, 15nh_{000019}, 15nt_{1}, 
\\16ah_{000006}, 16ah_{000009}, 16ah_{000010}, 16ah_{000019}, 16ah_{000060}, 16ah_{000062}, 16ah_{000069},
\\ 16ah_{000100}, 16ah_{000104}, 16ah_{000105}, 16ah_{000116}, 16ah_{000137}, 16ah_{000139}, 16ah_{000141},
\\ 16ah_{000153}, 16ah_{000209}, 16ah_{000253}, 16ah_{000447}, 16ah_{000505}, 16nh_{0000029}, 16nh_{0000034},
\\ 16nh_{0000092}, 16nh_{0000150}, 16nh_{0000415}, 17ah_{0000008}, 17ah_{0000010}, 17ah_{0000032}, 17ah_{0000081},
\\ 17ah_{0000112}, 17ah_{0000137}, 17ah_{0000138}, 17ah_{0000161}, 17ah_{0000170}, 17ah_{0000243},
17ah_{0000248},
\\ 17ah_{0000249}, 17ah_{0000341}, 17ah_{0000366}, 17ah_{0000368}, 17ah_{0000374}, 17ah_{0000376}, 17ah_{0000384},
\\ 17ah_{0000495}, 17ah_{0000500}, 17ah_{0000544}, 17ah_{0000545}, 17ah_{0000593}, 17ah_{0000634}, 17ah_{0000685},
\\ 17ah_{0000687}, 17ah_{0000786}, 17ah_{0000979}, 17ah_{0001175}, 17ah_{0001352}, 17ah_{0001734}, 17ah_{0001883},
\\ 17ah_{0002693}, 17ah_{0003282}, 17nh_{0000002}, 17nh_{0000034}, 17nh_{0000035}, 17nh_{0000196}, 17nh_{0000257},
\\ 17nh_{0000276}, 17nh_{0000327}, 17nh_{0000473}, 17nh_{0000765}, 17nh_{0005618}, 17nh_{0005619}
\end{eqnarray*}
\end{small}
Only two of those knots satisfy $\Delta_K''(1)=0$ and $J_K'''(1)=0,$ those are the knots $14nh_{00042}, 16ah_{000209}.$  We get the following corollary:
\begin{corollary}\label{cor:17crossings} No non-trivial knot with at most $17$ crossings admits purely cosmetic surgeries.
\end{corollary}
\begin{proof}
The two exceptions $14nh_{00042}$ and $16ah_{000209}$ that we found were covered by Hanselman's \cite{Han19} treatment of knots with at most $16$ crossings; for those two knots a purely cosmetic surgery can be excluded from the computation of $\widehat{HFK}.$ Moreover by \cite{Tao19}, a knot with a purely cosmetic surgery pair must be prime. 
\end{proof}
\begin{remark}Since $J_K''(1)=-3\Delta''(1)$ (see \cite[Lemma 2.1]{NW15}), computing only the Jones polynomial one can exclude all but two knots with at most $17$ crossings to have a purely cosmetic surgery. While we simply used Hanselman's previous results to exclude those two knots, we could have tried out the criterion in Theorem \ref{thm:main_thm} for $r=7$ to exclude those last two slopes instead.
\end{remark}
We note that Sikora and Tuzun \cite{ST18} have done extensive computations of Jones polynomials of knots with at most $22$ crossings, in order to verify the Jones unknot detection conjecture up to that crossing number. A numerical strategy to verify the conjecture up to some crossing number would be to test knots for the criterions $\Delta_K''(1)=0$ and $J_K'''(1)=0$ which are fast to check, if both compute whether $J_K(\zeta_5)=1,$ then for the remaining exceptions use the criterion for primes $r\geq 7.$ In forthcoming work, we plan to implement the criterion in Theorem \ref{thm:main_thm} for primes $r\geq 7$ to be able to verify the conjecture for the whole census of knots with at most $19$ crossings on Regina.

The paper is organized as follows: in Section \ref{sec:prelim}, we review the basics of the $\mathrm{SO}(3)$ WRT TQFTs that we will use. In Section \ref{sec:surgery_formula}, we derive the formula of the $\mathrm{SO}(3)$ WRT TQFTs of Dehn-surgeries of a knot from TQFT axioms. In Section \ref{sec:proofs}, we prove Theorem \ref{thm:main_thm2} using the surgery formula and deduce Theorem \ref{thm:main_thm}.

\textbf{Acknowledgements:} We would like to thank François Costantino, Andras Stipsciz and Effie Kalfagianni for their interest and helpful conservations. We also thank Cl\'ement Maria for helping us generate the census of Jones polynomials with Regina.
\section{Preliminaries}
\label{sec:prelim}
\subsection{$\mathrm{SO}_3$ WRT TQFTs and extended cobordisms}
\label{sec:extended}
In this section, we review some well-known properties of WRT TQFTs needed in this paper. Although they were initially defined by Reshetikhin and Turaev in \cite{RT91} to realize Witten's \cite{W} interpretation of the Jones polynomial as a quantum field theory based on the Chern-Simons invariant, it is really the so-called $\mathrm{SO}_3$ TQFTs in the framework of Blanchet, Habegger, Masbaum and Vogel \cite{BHMV} that we will present here. We note that the $3$-manifold invariant associated to the $\mathrm{SO}_3$ TQFT is up to normalization identical to the $\tau_r'$ invariants of Kirby and Melvin \cite{KM91}. 

The $\mathrm{SO}_3$ TQFTs in \cite{BHMV} have a so-called \textit{anomaly}, meaning that the associated representation of mapping class groups are projective. Lifting this anomaly requires to work with a category of $2+1$ cobordisms with an additional structure, either a $p_1$-structure as in \cite{BHMV}, or a so-called extended structure by an approach initiated by Walker \cite{W} and further developped by Turaev \cite{Tur}. We will use the later approach here.

 Let us define extended $3$-manifolds as pairs $(M,n)$ where $M$ is a compact oriented $3$-manifold and $n\in \Z.$ The integer $n$ is called the weight of the extended $3$-manifold, and intuitively encodes a choice of signature of a $4$-manifold bounding $M.$ A connected extended surface $(\Sigma,L)$ will consist of a compact oriented surface $\Sigma$ with a choice of Lagrangian $L\subset H_1(\Sigma,\QQ).$ Manifolds with boundary will be equiped with a weight and a Lagrangian in each boundary component. For any two extended $3$-manifolds $(M_1,n)$ and $(M_2,m)$ such that $(\Sigma,L)$ is the common boundary of $(M_1,n)$ and $(M_2,m),$ the gluing is the extended closed $3$-manifold $(M_1\underset{\Sigma}{\coprod} \overline{M_2},n+m-\mu(\lambda_{M_1}(\Sigma),L,\lambda_{M_2}(\Sigma)),$ where $\lambda_{M_i}(\Sigma)=\mathrm{Ker} H_1(\Sigma,\QQ)\rightarrow H_1(M,\QQ)$ and $\mu$ is the Maslov index, which can be computed from the following definition:
 \begin{definition}\label{def:Maslov}Let $(V,\omega)$ be a finite dimensional symplectic $\mathbb{Q}$-vector space, let $L_1,L_2,L_3$ be Lagrangians in $V$ and let $W=\lbrace (x,y) \in L_1\times L_2 \ | \ x+y\in L_3 \rbrace.$ Then the Maslov index $\mu(L_1,L_2,L_3)$ is the signature of the symmetric bilinear form $B$ on $W$ defined by $B((x_1,x_2),(y_1,y_2))=\omega(x_2,y_1).$
 \end{definition}
 We recall that the Maslov index changes sign under any odd permutation of $L_1,L_2,L_3.$ In particular, it vanishes whenever two Lagrangians are equal.
 
 Similar rules apply for the gluing of cobordisms, where the weight of the gluing has to computed using some Maslov indices. That construction applied to mapping cylinders gives rise to the so-called \textit{extended mapping class group} of surfaces, which we define below:
 \begin{definition}\label{def:extendedMCG} For $\Sigma$ a compact oriented surface, and $L\subset H_1(\Sigma,\QQ)$ a Lagrangian, the extended mapping class group $\widetilde{\mathrm{Mod}}(\Sigma)$ is the $\mathbb{Z}$-central extension of $\mathrm{Mod}(\Sigma)$ defined by the rule
 $$(f,n)\circ (g,m)=(f\circ g, n+m+\mu(L,f(L),(f\circ g)(L)))$$
 for any $f,g\in \mathrm{Mod}(\Sigma)$ and any $n,m\in \Z.$
 \end{definition}
\subsection{The $\mathrm{SO}_3$ invariants from skein theory}
\label{sec:invariants}
Here we review the TQFT constructions defined in \cite{BHMV}. We will first give the definition of the $\mathrm{SO}_3$ invariants of an extended closed $3$-manifold $(M,n)$, which are computed from the evaluation at roots of unity of the Kauffman bracket of some cablings of a surgery presentation of $M.$

 Let us fix an odd integer $r\geqslant 3,$ and $A_r$ be a primitive $2r$-th root of unity. For any integer $n\in \Z$ the quantum integer $[n]$ is given by
$$[n]=\frac{A_r^{2n}-A_r^{-2n}}{A_r^2-A_r^{-2}}$$
\begin{figure}\label{fig:Kauffman}
\begingroup%
  \makeatletter%
  \providecommand\color[2][]{%
    \errmessage{(Inkscape) Color is used for the text in Inkscape, but the package 'color.sty' is not loaded}%
    \renewcommand\color[2][]{}%
  }%
  \providecommand\transparent[1]{%
    \errmessage{(Inkscape) Transparency is used (non-zero) for the text in Inkscape, but the package 'transparent.sty' is not loaded}%
    \renewcommand\transparent[1]{}%
  }%
  \providecommand\rotatebox[2]{#2}%
  \newcommand*\fsize{\dimexpr\f@size pt\relax}%
  \newcommand*\lineheight[1]{\fontsize{\fsize}{#1\fsize}\selectfont}%
  \ifx\svgwidth\undefined%
    \setlength{\unitlength}{154.39388606bp}%
    \ifx\svgscale\undefined%
      \relax%
    \else%
      \setlength{\unitlength}{\unitlength * \real{\svgscale}}%
    \fi%
  \else%
    \setlength{\unitlength}{\svgwidth}%
  \fi%
  \global\let\svgwidth\undefined%
  \global\let\svgscale\undefined%
  \makeatother%
  \begin{picture}(1,0.36409275)%
    \lineheight{1}%
    \setlength\tabcolsep{0pt}%
    \put(0,0){\includegraphics[width=\unitlength,page=1]{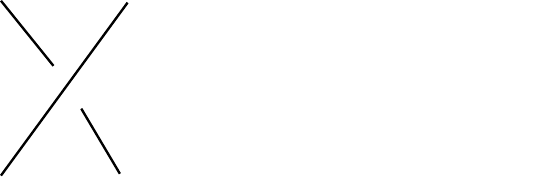}}%
    \put(0.24844794,0.18227415){\makebox(0,0)[lt]{\lineheight{1.25}\smash{\begin{tabular}[t]{l}$=A$\end{tabular}}}}%
    \put(0,0){\includegraphics[width=\unitlength,page=2]{Kauffman_rel.pdf}}%
    \put(0.63706435,0.16839498){\makebox(0,0)[lt]{\lineheight{1.25}\smash{\begin{tabular}[t]{l}$+A^{-1}$\end{tabular}}}}%
    \put(0,0){\includegraphics[width=\unitlength,page=3]{Kauffman_rel.pdf}}%
  \end{picture}%
\endgroup%

\caption{The second Kauffman relation}
\end{figure}
Let us first recall the definition of the Kauffman bracket $\langle L \rangle$ of a framed link in $S^3.$ It is an invariant of framed links completely determined by the normalization $\langle \emptyset \rangle=1,$ and the two Kauffman relations: $\langle L \cup U \rangle=(-A^2-A^{-2})\langle L \rangle,$ where $L \cup U$ denotes the disjoint union of a framed link and a $0$-framed unknot, and the second Kauffman relation relating $3$ links that differ only in a ball is shown in Figure \ref{fig:Kauffman}.
\\ Furthermore, the \textit{colored} Kauffman bracket $\langle L,c\rangle$ of a link $L$ whose components are colored by elements of $\C[z]$ is defined in the following way: if all components $L_i$'s are colored by monomials $z^{d_i}$ then just replace the component $L_i$ by $d_i$ parallel copies and compute the Kauffman bracket. Otherwise, expand by multilinearity.
\\ Now, suppose that $M$ has a surgery presentation $M=S^3(L),$ when $L$ is a framed link in $S^3.$ Let us write $r=2m+1,$ and let $\omega=\underset{i=1}{\overset{m}{\sum}}(-1)^{i-1}[i]e_i,$ when $e_i(z)$ is the $i$-th Chebyshev polynomial, defined by $e_1(z)=1, e_2(z)=z$ and $e_{i+1}(z)=ze_i(z)-e_{i+1}(z).$ We recall that $[i]=\frac{A_r^{2i}-A_r^{-2i}}{A_r^2-A_r^{-2}}$ is the $i$-th quantum integer.
\\ Finally, we introduce 
$$\eta_r=\frac{(A_r^2-A_r^{-2})}{\sqrt{-r}} \ \textrm{and} \ \kappa_r=\eta_r\langle U_-,\omega\rangle=\eta_r \underset{i=1}{\overset{\frac{r-1}{2}}{\sum}} (-A_r)^{-(i^2-1)}[i]^2,$$
where $U_-$ is the unknot with framing $-1.$ We note that $\kappa_r$ is actually a $4r$-th root of unity, whose exact order depends on $r \ \mathrm{mod} \ 4$ and the choice of $A_r.$
\begin{theorem}\label{thm:RT_invariants}\cite[Theorem 2.2.2]{Tur}
If $M$ is obtained by surgery on a framed link $L\subset S^3,$ with $n(L)$ components and signature $\sigma(L)$ then 
$$Z_r(M,\sigma(L))=\eta_r^{1+n(L)}\langle L, \omega,\omega,\ldots,\omega \rangle$$
is a topological invariant of the extended $3$-manifold $(M,\sigma(L)).$ 

Furthermore if, $M$ contains a colored link $(K,c)$ then
$$Z_r(M,\sigma(L),K,c)=\eta_r^{1+n(L)}\langle L\cup K,\omega,\ldots,\omega,c \rangle.$$
is a topological invariant of the extended $3$-manifold $(M,\sigma(L),K,c).$

In both cases, the invariant satisfies $Z_r(M,n)=\kappa_r^{-n} Z_r(M,0).$ 
\end{theorem} 
Although we defined the invariant for extended $3$-manifolds, for any $3$-manifold the quantity $Z_r(M)=Z_r(M,0)$ is a topological invariant of $M.$ The last property allows one to compute $Z_r(M)$ from any surgery presentation of $M,$ no matter the signature, by renormalizing by the appropriate power of $\kappa_r.$

\subsection{The $\mathrm{SO}_3$ WRT TQFTs}
\label{sec:TQFT}
 The $\mathrm{SO}_3$ WRT invariants $Z_r$ are part of a TQFT defined on the category of extended cobordisms in dimension $2+1.$ Objects of this category are extended surfaces and morphisms are (homeomorphism classes fixing the boundary of) extended cobordism of dimension $3,$ as introduced in Section \ref{sec:extended}. The disjoint union gives the category a monoidal structure, and the invariant $Z_r$ can be extended to a monoidal functor from the category of extended $2+1$-cobordisms to the category of complex vector spaces. Concretely speaking, we have the following:
\begin{theorem}\cite{BHMV}\label{thm:TQFT} For any odd integer $r\geqslant 3,$ we have:
\begin{itemize}
\item[-]For any closed compact oriented $3$-manifold $M,$ $Z_r(M)=Z_r(M,0)$ is a $\C$-valued topological invariant. 
\item[-]For every closed compact oriented extended surface $(\Sigma,L),$ $Z_r(\Sigma,L)$ is a finite dimensional vector space with a natural Hermitian form $\langle\langle \cdot , \cdot \rangle\rangle.$
\item[-]For any compact oriented extended $3$-manifold $M=(M,0),$ containing a link $K\subset M,$ and with a fixed homeomorphism $\partial M \simeq \Sigma,$ and a choice of Lagrangian $L\subset H_1(\Sigma,\QQ),$ $Z_r(M,K)$ is a vector in $Z_r(\Sigma,L),$ and $Z_r(\Sigma,L)$ is spanned by such vectors.
\item[-]The extended mapping class group $\widetilde{\mathrm{Mod}}(\Sigma),$ acting on extended $3$-manifolds with boundary $(\Sigma,L)$ gives rise to a representation 
$$\rho_r:\widetilde{\mathrm{Mod}}(\Sigma) \longmapsto \mathrm{Aut}(Z_r(\Sigma,L))$$
called the $\mathrm{SO}_3$ quantum representation.
\item[-]For any two extended $3$-manifolds (possibly with links) $(M_1,0),(M_2,0)$ with $\partial M_1\simeq \partial M_2 \simeq (\Sigma,L)$ the underlying closed $3$-manifold $M=M_1 \underset{\Sigma}{\coprod} \overline{M_2}$ has invariant
$$Z_r(M)=\kappa_r^{\mu(\lambda_{M_1}(\Sigma),L,\lambda_{M_2}(\Sigma))}\langle\langle Z_r(M_1) , Z_r(M_2) \rangle\rangle.$$
\end{itemize}
\end{theorem}
In this paper, we will only compute WRT invariants of manifolds that are Dehn surgery of knots rather than links, and will use the TQFT properties of $Z_r,$ to compute them. That, we will use the last point of Theorem \ref{thm:TQFT} in the special case where $\Sigma$ is a torus, $M_1$ is a knot complement in $S^3$ and $M_2$ is a solid torus. Therefore we will now focus on the TQFT space and quantum representation of the torus.
\subsection{Quantum representations of the torus}
\label{sec:quantum_rep}
We first describe a basis of the $Z_r$ TQFT-space of the torus $\mathbb{T}^2.$ We will fix the choice of Lagrangian in $H_1(\mathbb{T}^2,\QQ)$ to be the subspace generated by the class of the meridian $S^1\times \lbrace 0 \rbrace,$ and therefore no longer make reference of the choice of Lagrangian in this section.

If $1\leqslant i \leqslant \frac{r-1}{2},$ let $f_i$ be the vector in $Z_r(\mathbb{T}^2)$ corresponding to the solid torus $D^2\times S^1$ with boundary $T^2$ and meridian $S^1 \times \lbrace 0 \rbrace,$ containing the framed knot $\lbrace [0,\varepsilon] \rbrace \times S^1$ colored by $e_i(z).$ In particular, $f_1$ corresponds to the solid torus with the empty link inside, since $e_1(z)=1.$
\begin{proposition}\cite[Corollary 4.10]{BHMV} Let $r=2m+1 \geqslant 3$ be an odd integer. The vector space $Z_r(\mathbb{T}^2)$ admits $f_1,\ldots,f_{m}$ as an orthonormal basis.
\end{proposition}
With this in mind, we will now describe the quantum representations of the torus, in the above basis.
\\ Let us recall that the mapping class group of the torus is isomorphic to $\mathrm{SL}_2(\Z),$ and generated by the two matrices $T=\begin{pmatrix}
1 & 1 \\ 0 & 1
\end{pmatrix}$ and $S=\begin{pmatrix}
0 & -1 \\ 1 & 0
\end{pmatrix}.$ As mapping classes, $T$ corresponds to the Dehn-twist along the meridian $S^1 \times \lbrace 0 \rbrace$ of $\mathbb{T}^2=S^1 \times S^1,$ and $S$ to the map $S(u,v)=(-v,u)$ of order $4.$ A presentation of $\mathrm{SL}_2(\Z)$ is then given by $$\mathrm{SL}_2(\Z)=\langle S,T | S^4=1,S^2=(ST)^3 \rangle.$$
In the extended mapping class group, the relations turn into
$$(S,0)^4=(Id,0), \  (S,0)^2=(S^2,0) \ \textrm{and} \ ((S,0)(T,0))^3=(S^2,1)$$
Indeed, let $a,b$ be the basis of $H_1(\mathbb{T}^2,\QQ)$ given by the class of the meridian $S^1\times \lbrace 0 \rbrace$ and longitude $\lbrace 0 \rbrace \times S^1.$ We have the following quick rule to compute Maslov indices in the case of the torus:
\begin{lemma}\label{lemma:Maslov}Let $L_1=\mathrm{Span}(x),$ $L_2=\mathrm{Span}(y)$ and $L_3=\mathrm{Span}(z)$ be Lagrangians of $H_1(\mathbb{T}^2,\QQ).$ Then if any two of $x,y,z$ are linearly dependent then $\mu(L_1,L_2,L_3)=0$ and else if $z=\alpha x+\beta y$ then 
$$\mu(L_1,L_2,L_3)=\mathrm{sign}(\alpha\beta \omega(x,y)).$$
\end{lemma}
\begin{proof}
Follows from Definition \ref{def:Maslov}.
\end{proof}
To simplify notations, we will write $\mu(x,y,z)$ for $\mu(\mathrm{Span}(x),\mathrm{Span}(y),\mathrm{Span}(z)).$
We can compute that $(S,0)^2=(S^2,\mu(a,b,a))=(S^2,0)$ and $(S,0)^4=(S^4,\mu(a,a,a))=(Id,0).$ Moreover, $(S,0)(T,0)=(ST,\mu(a,b,a))=(ST,0),$ $(ST,0)^2=((ST)^2,\mu(a,b,-a+b))=((ST)^2,-1),$ and $(ST,0)^3=((ST)^3,-1+\mu(a,-a+b,-a))=(S^2,-1).$
\begin{proposition}\cite[Section 2]{Gil99}\label{prop:quantum_rep} We have, in the basis $e_i,$ that
$$\rho_r((T,0))f_i=(-A_r)^{i^2-1}f_i, \ \textrm{and} \ \rho_r((S,0))f_i=\eta_r \underset{1\leqslant j \leqslant \frac{r-1}{2}}{\sum}(-1)^{i+j}[ij] f_j$$
and moreover $\rho_r((Id,1))=\kappa_r Id.$
\end{proposition}
\begin{remark} \label{rk:eigenvalues} Note that the matrix $\rho_r(T)$ is diagonal with distinct eigenvalues, since $(-A_r)$ is a primitive $r$-th root of $1.$
\end{remark}
To simplify notations, we will sometimes abusively write $\rho_r(S^{n_1}T^{n_2}\ldots S^{n_k})$ for

\noindent $\rho_r((S,0)^{n_1}(T,0)^{n_2}\ldots (S,0)^{n_k}),$ even though $\rho_r$ is defined on $\widetilde{\mathrm{Mod}}(\mathbb{T}^2)$ instead of $\mathrm{SL}_2(\Z).$ If the word in the generators $S$ and $T$ is fixed then there is no ambiguity.

The following lemma shows that the above definition yields a projective representation of $\mathrm{SL}_2(\Z),$ or a linear representation of the extended mapping class group $\widetilde{\mathrm{Mod}}(\mathbb{T}^2):$
\begin{lemma}\label{lemma:quantum_rep} The matrices $\rho_r(S)$ and $\rho_r(T)$ are unitary for $\langle\langle,\rangle\rangle,$ and we have $\rho_r(T^r)=Id,$ $\rho_r(S^2)=Id$ and $\rho_r((ST)^3)=\kappa_r^{-1} Id.$
\end{lemma}
\begin{proof}
We refer to \cite[Section 3.9]{Tur} for a proof. Note that our constants $\eta_r$ and $\kappa_r$ correspond to the constants $\mathcal{D}^{-1}$ and $\Delta \mathcal{D}^{-1},$ and our matrices $\rho_r((S,0))$ and $\rho_r((T,0))$ are the same as the matrices $\mathcal{D}^{-1}S$ and $T$ in \cite{Tur}. Finally $J=\mathrm{Id}$ in our context. Then in \cite{Tur} the equivalent relation $\rho_r(STS)=\kappa_r^{-1}\rho_r(T^{-1}ST^{-1})$ is proven.
\end{proof}
 
\section{Surgery formulas for $Z_r$}
\label{sec:surgery_formula}
In this section, we will present formulas that express the $Z_r$ invariant of a Dehn-surgery on a knot in terms of its colored Jones polynomials. We will first describe the surgery formulas for $Z_r$ for arbitrary slopes, then we will specialize to slopes of the form $\frac{1}{k}, k\in \Z$ or $\pm 2.$ We have the following:
\begin{proposition}\label{prop:surg_formula1} Let $K$ be a framed knot in $S^3,$ and let $r=2m+1\geqslant 3$ be an odd integer. Then, in the basis $e_1,\ldots,e_m,$ we have
$$Z_r(E_K)=Z_r(E_K,0)=\eta_r\begin{pmatrix}
1 \\ \langle K,2 \rangle
\\ \langle K,3 \rangle
\\ \vdots
\\ \langle K,m\rangle
\end{pmatrix}$$
Moreover, if $\phi$ is any element of $\mathrm{Mod}(\partial E_K)$ which sends the meridian of $K$ to the curve of slope $s,$ then
$$Z_r(E_K(s),0)=\kappa_r^{-\mathrm{sign}(s)}\langle\langle Z_r(E_K),\rho_r(\phi,0)e_1 \rangle\rangle,$$
where by definition $\mathrm{sign}(s)=0$ for $s\in \lbrace 0,\infty\rbrace.$
\end{proposition}
\begin{proof}
The first identity expresses the fact that pairing $E_K$ with the basis vectors $e_i,$ we simply get the manifold $S^3$ with the link $K$ colored by $i$ in it. By Theorem \ref{thm:RT_invariants}, $\eta_r$ is the $Z_r$ invariant of $S^3,$ and adding a knot $K$ colored by $e_i$ multiplies the invariant by $\langle K,e_i  \rangle.$
\\ As for the second identity, recall that $f_1$ is the vector corresponding to the solid torus (with meridian  the meridian of $K$). Thus the second identity follows from the last part of Theorem \ref{thm:TQFT} and the definition of the quantum representation, once we show that the Maslov index is $-\mathrm{sign}(s).$

Since we equip the torus with the span $\mathrm{Span}(a)$ of the meridian as Lagrangian, and the longitude $b$ bounds a surface in $E_K$ while the curve of slope $s=p/q$ bounds a disk in the solid torus, this Maslov index is
$$\mu(b,a,pa+qb)=-\mathrm{sign}(p/q)=-\mathrm{sign}(s).$$
In the cases $s=0$ or $s=\infty$ the Maslov index vanishes.  
\end{proof}
\begin{remark}One may want to compare the vector $Z_r(E_K)$ with the vector $v_K$ in Theorem \ref{thm:main_thm2}. To do this, we should say that while the colored Jones polynomials are invariants of knots, the colored Kauffman brackets are invariants of framed knots only. However when talking about the cosmetic surgery problem for a knot $K$, a framing on $K$ is somewhat implicit, if we want to be able to talk about slopes on the knot.
\\ Taking the convention that we chose as framing for $K$ the longitude with zero winding number, we have for any $A\in \C,$ 
$$\langle K,e_i\rangle(A_r)=(-1)^{i-1}[i]J_{K,i}(A_r^4)=(-1)^{i-1}[i]J_{K,i}(\zeta_r^2)$$
where $[i]=\frac{A_r^{2i}-A_r^{-2i}}{A_r^2-A_r^{-2}}=\frac{\zeta_r^i-\zeta_r^{-i}}{\zeta_r-\zeta_r^{-1}}.$
\end{remark}
\begin{remark}The map $\phi\in \mathrm{Mod}(\mathbb{T}^2)$ which sends the meridian to the curve of slope $s$ is not unique. However, any two such maps differ by multiplication on the right by $T^k,$ with $k\in \Z.$ As $\rho_r(T)f_1=f_1,$ the pairing $\langle Z_r(E_K),\rho_r(\phi)f_1 \rangle$ does not depend on the choice of $\phi.$
\end{remark}
\begin{proposition}\label{prop:surgery_formula2} Let $K$ be a knot in $S^3$ then we have
$$Z_r(E_K(\tfrac{1}{k}),0)=\langle\langle Z_r(E_K),\rho_r(ST^{-k}S)f_1 \rangle\rangle$$
and 
$$Z_r(E_K(2),0)=\kappa_r^{-1}\langle\langle Z_r(E_K),\rho_r(T^2S)f_1 \rangle\rangle$$
$$Z_r(E_K(-2),0)=\kappa_r \langle\langle Z_r(E_K),\rho_r(T^{-2}S)f_1 \rangle\rangle$$
\end{proposition}

We note that writing the continued fraction expansion of $s=p/q,$ one can find a word in the generators $S$ and $T$ that maps the meridian to the curve of slope $s.$ Giving similar surgery formula then amounts to computing the contribution $\kappa_r^{m_s}$ coming from Maslov indices. Explicit formulas for $m_s$ can given in terms of Dedekind sums or Rademacher phi functions, for example closely related computations can be found in \cite{Jef92}, where the $Z_r$ invariant of lens spaces is computed.

Moreover, closely related surgery formulas for the $\tau_r'$ invariant for the integer and $1/k$-surgeries were computed in \cite{KM91}.

\begin{proof}First, let us show that $(S,0)(T,0)^{-k}(S,0)=(ST^{-k}S,\mathrm{sign}(k))$ and $(T,0)^{\pm 2}(S,0)=(T^{\pm 2}S,0).$ We compute that $(T,0)^{k}=(T^k,0)$ for any $k\in \Z$ by induction as $\mu(a,a,a)=0.$ Moreover,
$$(S,0)(T^{-k},0)=(ST^{-k},\mu(a,b,b))=(ST^{-k},0),$$
$$(ST^{-k},0)(S,0)=(ST^{-k}S,\mu(a,b,-a-kb))=(ST^{-k}S,\mathrm{sign}(k)),$$
and
$$(T^{\pm 2},0)(S,0)=(T^{\pm 2}S,\mu(a,a,\mp 2a-b))=(T^{\pm 2}S,0).$$
Now, by Proposition \ref{prop:surg_formula1}, we have 
\begin{multline*}Z_r(E_K(\tfrac{1}{k}))=\kappa_r^{-\mathrm{sign}(k)}\langle\langle Z_r(E_K),\rho_r(ST^{-k}S,0)f_1 \rangle\rangle
\\=\kappa_r^{-\mathrm{sign}(k)}\langle\langle Z_r(E_K),\kappa_r^{-\mathrm{sign}(k)}\rho_r(ST^{-k}S)f_1 \rangle\rangle=\langle\langle Z_r(E_K),\rho_r(ST^{-k}S)f_1 \rangle\rangle
\end{multline*}
where the last equality uses that $\langle\langle,\rangle\rangle$ is anti-linear on the right.

Similarly, we have
$$Z_r(E_K(2))=\kappa_r^{-1}\langle\langle Z_r(E_K),\rho_r(T^{2}S,0)f_1 \rangle\rangle=\kappa_r^{-1}\langle\langle Z_r(E_K),\rho_r(T^{2}S)f_1 \rangle\rangle$$
and 
$$Z_r(E_K(-2))=\kappa_r\langle\langle Z_r(E_K),\rho_r(T^{-2}S,0)f_1 \rangle\rangle=\kappa_r\langle\langle Z_r(E_K),\rho_r(T^{-2}S)f_1 \rangle\rangle.$$
\end{proof}
\section{Proof of the main theorems}
\label{sec:proofs}
Before we move to the proof of Theorem \ref{thm:main_thm2}, we will explicit the finite set $F_r$ of vectors in $Z_r(\mathbb{T}^2).$
\begin{definition}\label{def:finite_set} Let $r\geqslant 5$ be a prime. We define
$$F_r=\lbrace \rho_r(ST^{-k}S)f_1-\rho_r(ST^kS)f_1, 1\leq k \leq \tfrac{r-1}{2} \rbrace \cup \lbrace \rho_r(T^2S)f_1-\kappa_r^{-2}\rho_r(T^{-2}S)f_1 \rbrace$$ where $\rho_r(T),\rho_r(S)$ are the matrices defined in Proposition \ref{prop:quantum_rep} and $\kappa_r$ is the constant defined in Proposition \ref{lemma:quantum_rep}.
\end{definition}
It is clear from the definition that $|F_r|\leq \frac{r+1}{2}.$ The following lemma shows that $F_r$ contains only non-zero vectors:
\begin{lemma}\label{lemma:finite_set}Let $r\geq 5$ be an odd prime. Then $\rho_r(ST^{-k}S)f_1-\rho_r(ST^kS)f_1=0$ if and only if $k=0 \ \mathrm{mod} \ r.$
Moreover, $\rho_r(T^2S)f_1-\kappa_r^{-2}\rho_r(T^{-2}S)f_1\neq 0.$
\end{lemma} 
\begin{proof}
Note that as $\rho_r(T^r)=Id,$ if $k=0 \ \mathrm{mod} \ r,$ then $\rho_r(ST^{\pm k}S)=\rho_r(S^2)=Id$ and thus $\rho_r(ST^{-k}S)f_1-\rho_r(ST^kS)f_1=0.$
Now, by contradiction assume $\rho_r(ST^{-k}S)f_1=\rho_r(ST^kS)f_1$ but $k\neq 0 \ \mathrm{mod}\ r.$ Then applying $\rho_r(T^kS)$ on the left we get $\rho_r(T^{2k}S)f_1=\rho_r(S)f_1.$ So $\rho_r(S)f_1$ must be an eigenvector of $\rho_r(T^{2k})$ of eigenvalue $1.$ As the order of $r$ is coprime with $2k,$ we deduce that $\rho_r(S)f_1$ must be an eigenvector of $\rho_r(T)$ of eigenvalue $1,$ thus $\rho_r(S)f_1$ must be colinear to $f_1$ by Remark \ref{rk:eigenvalues}. This is not the case: as all quantum integers $[i]=\frac{A_r^{2i}-A_r^{-2i}}{A_r^2-A_r^{-2}}$ with $1\leq i \leq \tfrac{r-1}{2}$ are non-zero, $\rho_r(S)f_1$ has non-zero coefficient along each $f_i.$ Note that we use here that $\mathrm{dim}(Z_r(\mathbb{T}^2))\geq 2$ for $r\geq 5.$

Now, by contradiction assume that $\rho_r(T^2S)f_1=\kappa_r^{-2}\rho_r(T^{-2}S)f_1,$ then similarly $\rho_r(S)f_1$ must be an eigenvector of $\rho_r(T^4),$ and thus an eigenvector of $\rho_r(T)$ as $4$ is coprime to $r$ which is the order of $T.$ Note that the diagonal coefficients $(-A_r)^{i^2-1}$ where $i=1,\ldots ,\frac{r-1}{2}$ of $\rho_r(T)$ are all distinct (see Remark \ref{rk:eigenvalues}.) Thus $\rho_r(S)f_1$ would have to be colinear to one of the $f_i$'s, which is not the case.
\end{proof}
\begin{remark}\label{rk:order3root} The proof above uses the fact that $r>3.$ In the case of $r=3,$ we would find that $\mathrm{dim}(Z_3(\mathbb{T}^2)=1,$ $\kappa_3=1, \rho_3(S)=\eta_3=1$ and $\rho_3(T)=1.$ This implies that when $r=3,$ all of the vectors in Definition \ref{def:finite_set} are zero.
\end{remark}
We are now ready to prove Theorem \ref{thm:main_thm2}:
\begin{proof}[Proof of Theorem \ref{thm:main_thm2}]Fix $r\geq 5$ an odd prime.
Let $K$ be a knot in $S^3$ with a purely cosmetic surgery pair $s>s'.$ By Theorem \ref{thm:hanselman}, the slopes $s,s'$ must be opposite and $s=\frac{1}{k}$ or $s=2.$

Let us consider the first case and assume that $r$ does not divide $k.$ We have $E_K(\tfrac{1}{k})\simeq E_K(-\tfrac{1}{k}),$ thus the $Z_r$ invariants coincide: $Z_r(E_K(\tfrac{1}{k}))=Z_r(E_K(-\tfrac{1}{k})).$  By Proposition \ref{prop:surgery_formula2} we have
$$\langle\langle Z_r(E_K),\rho_r(ST^{-k}S)f_1-\rho_r(ST^kS)f_1 \rangle\rangle=Z_r(E_K(\tfrac{1}{k}))-Z_r(E_K(-\tfrac{1}{k}))=0.$$
Note that $\rho_r(ST^{\pm k}S)f_1$ depends only on $k$ mod $r$ as $\rho_r(T^r)=Id.$ Let $l \in \lbrace \pm 1, \pm 2, \ldots , \pm \frac{r-1}{2}\rbrace$ such that $k=l \ \mathrm{mod} \ r.$ Then $Z_r(E_K)$ must be orthogonal to the vector $\rho_r(ST^{-l}S)f_1-\rho_r(ST^lS)f_1.$ Either this vector or its opposite is in the set $F_r,$ so the claim is proved in that case.

Similarly, in the case where $E_K(2)\simeq E_K(-2),$ applying Proposition \ref{prop:surgery_formula2}, we get that 
$$\langle\langle Z_r(E_K),\rho_r(T^2S)f_1-\kappa_r^{-2}\rho_r(T^{-2}S)f_1 \rangle\rangle=\kappa_r\left(Z_r(E_K(2))-Z_r(E_K(-2))\right)=0,$$
and thus $Z_r(E_K)$ is orthogonal to $\rho_r(T^2S)f_1-\kappa_r^{-2}\rho_r(T^{-2}S)f_1,$ which is in $F_r$ by definition. 
\end{proof}
\begin{remark}\label{rk:oneSlope} The proof of Theorem \ref{thm:main_thm} makes clear that if $k=\pm l \ \mathrm{mod} \ r$ with $1\leq l \leq \tfrac{r-1}{2}$ and $\lbrace \pm \frac{1}{k}\rbrace$ is a purely cosmetic surgery pair for a knot $K,$ then $Z_r(E_K)$ is orthogonal to a specific vector in $F_r:$ namely the vector $\rho_r(ST^{-l}S)f_1-\rho_r(ST^lS)f_1.$ Similarly, if $\lbrace \pm 2 \rbrace$ is a purely cosmetic surgery pair then $Z_r(E_K)$ is orthogonal to $\rho_r(T^2S)f_1-\kappa_r^{-2}\rho_r(T^{-2}S)f_1.$ 
\end{remark}

\begin{remark}\label{rk:unknot}Note that the unknot admits $\pm \frac{1}{k}$ cosmetic surgery pair for any $k\in \Z,$ and that $\pm 2$ is also a purely cosmetic surgery pair for the unknot. Therefore the vector $Z_r(E_U)$ is orthogonal to all of the vectors in $F_r.$ As $Z_r(E_U)=\rho_r(S)f_1$ is a non-zero vector, all of the vectors in $F_r$ actually lie in a codimension $1$ subspace of $Z_r(\mathbb{T}^2).$
\end{remark}
The proof of Theorem \ref{thm:main_thm} is a simple corollary of Theorem \ref{thm:main_thm2}, Remark \ref{rk:unknot} and the fact that if $r=5,$ then the dimension of $Z_r(\mathbb{T}^2)$ is $2:$
\begin{proof}[Proof of Theorem \ref{thm:main_thm}]
By Theorem \ref{thm:main_thm2}, if $K$ has a purely cosmetic surgery pair $\pm s$ then either the slope $s$ is of the form $\frac{1}{5k}$ or $Z_5(E_K)$ must be orthogonal to one of the (non-zero) vectors in $F_5.$ Note that the vector of the unknot $Z_5(E_U)=\rho_r(S)f_1\neq 0$ is orthogonal to all of the vectors in $F_5$ by Remark \ref{rk:unknot}. As $\mathrm{dim}(Z_5(\mathbb{T}^2))=\frac{5-1}{2}=2,$ the vectors in $F_5$ are all colinear and any vector orthogonal to a vector in $F_5$ must be colinear to $Z_5(E_U).$

So if the slope $s$ is not of the form $\frac{1}{5k},$ the vector $Z_r(E_K)=\eta_5\begin{pmatrix}
1 \\ -[2]J_K(\zeta_5^2)
\end{pmatrix}$ must be colinear to $Z_r(E_U)=\eta_5\begin{pmatrix}
1 \\ -[2]
\end{pmatrix}.$

 As $\eta_5$ and $[2]$ are non-zero, this implies that $J_K(\zeta_5^2)=1.$ By Galois action, we also have $J_K(\zeta_5)=1.$
\end{proof}
\def\cprime{$'$}
\providecommand{\bysame}{\leavevmode\hbox to3em{\hrulefill}\thinspace}
\providecommand{\href}[2]{#2}

\end{document}